\documentclass[12pt, leqno]{amsart}

\usepackage[OT2,T1]{fontenc}
\DeclareSymbolFont{cyrletters}{OT2}{wncyr}{m}{n}
\DeclareMathSymbol{\Sha}{\mathalpha}{cyrletters}{"58}

\usepackage{indentfirst}
\usepackage{amstext}
\usepackage{amsopn}
\usepackage{amsfonts}
\usepackage{amsmath}
\usepackage{latexsym}
\usepackage{amscd}
\usepackage{amssymb}
\usepackage{amsmath}
\usepackage[all,cmtip]{xy}
\usepackage{leftidx}
\usepackage{graphicx}
\usepackage{tikz}

=\oddsidemargin \font\teneufm=eufm10 \font\seveneufm=eufm7
\font\fiveeufm=eufm5
\newfam\eufmfam
\textfont\eufmfam=\teneufm \scriptfont\eufmfam=\seveneufm
\scriptscriptfont\eufmfam=\fiveeufm

\let\goth\mathfrak
\def\cA{\mathcal A}
\def\cB{\mathcal B}
\def\cC{\mathcal C}

\def\cF{\mathcal F}

\def\cH{\mathcal H}

\def\cO{\mathcal O}

\def\cE{\mathcal E}

\def\cX{\mathcal X}

\def\GG{\mathbb{G}}

\def\WW{\mathbf{W}}
\def\VV{\mathbf{V}}

\def\gG{\goth G}
\def\gP{\goth P}
\def\gQ{\goth Q}
\def\gH{\goth H}
\def\gM{\goth M}
\def\gT{\goth T}

\def\gX{\goth X}

\def\1{\mbox{\bf 1}}

\def\rad{\mathrm{rad}}

 \DeclareMathOperator{\Hom}{Hom}
\DeclareMathOperator{\Aut}{Aut}

\DeclareMathOperator{\Out}{Out} 
\DeclareMathOperator{\Isom}{Isom}
\DeclareMathOperator{\Isomext}{Isomext}

\DeclareMathOperator{\End}{End} 

\DeclareMathOperator{\GL}{\rm GL}

\newcommand{\incl}[1][r]
{\ar@<-0.2pc>@{^(-}[#1] \ar@<+0.2pc>@{-}[#1]}

\newcommand{\imm}[1][r]
   {\ar@{}[#1] |*[o][F]{\hbox{%
         }}
     \ar@{^{(}->}[#1]}

\newtheorem{stheorem}{Theorem}[section]%Theorems et al in italic font.

\newtheorem{scorollary}[stheorem]{Corollary}
\newtheorem{slemma}[stheorem]{Lemma}
\newtheorem{sproposition}[stheorem]{Proposition}
\newtheorem{sremark}[stheorem]{Remark}

\newtheorem{sexamples}[stheorem]{Examples}
\newtheorem{sdefinition}[stheorem]{Definition}

\theoremstyle{definition}%Theorems et al in roman font.

\numberwithin{equation}{section}

\def\ZZ{\mathbb{Z}}

\def\gG{\mathfrak{G}}

\def\gP{\mathfrak{P}}
\def\gQ{\mathfrak{Q}}

\def\gS{\mathfrak{S}}

\def\cO{\mathcal{O}}

\def\2int{\mathop{2\int}\nolimits}

\def\End{\mathop{\rm  End}\nolimits}
\def\rank{\mathop{\rm rank}\nolimits}

\def\Spec{\mathop{\rm Spec}\nolimits}

\def\Hom{\mathop{\rm Hom}\nolimits}

\def\Mat{\mathop{\rm M}\nolimits}

\def\Aut{\text{\rm{Aut}}}
\def\Out{\text{\rm{Out}}}

\def\Isom{\mathop{\rm Isom}\nolimits}
\def\Isomext{\mathop{\rm Isomext}\nolimits}

\def\resp.{\mathop{\rm resp.}\nolimits}

\def\lgr{\longrightarrow}

\font\math=cmmi10
\def\varpi{\hbox{\math\char'44}}

\def\simlgr{\buildrel\sim\over\lgr}

\def\pa{\S\kern.15em }

\def\un{\uppercase\expandafter{\romannumeral 1}}
\def\deux{\uppercase\expandafter{\romannumeral 2}}
\def\trois{\uppercase\expandafter{\romannumeral 3}}
\def\quatre{\uppercase\expandafter{\romannumeral 4}}
\def\cinq{\uppercase\expandafter{\romannumeral 5}}
\def\six{\uppercase\expandafter{\romannumeral 6}}

\def\hfl#1#2#3{\smash{\mathop{\hbox to#3{\rightarrowfill}}\limits
^{\scriptstyle#1}_{\scriptstyle#2}}}
\def\gfl#1#2#3{\smash{\mathop{\hbox to#3{\leftarrowfill}}\limits
^{\scriptstyle#1}_{\scriptstyle#2}}}

\title[Reductive group schemes]{When is a reductive group scheme linear?}

\date{\today}

\author{Philippe Gille}

\address[]{P. Gille, Institut Camille Jordan - Universit\'e Claude Bernard Lyon 1
43 boulevard du 11 novembre 1918,
69622 Villeurbanne cedex - France }
\email{gille@math.univ-lyon1.fr}
\thanks{The author is  supported  by the project ANR Geolie, ANR-15-CE
40-0012, (The French National Research Agency).}

\begin{document}

 \begin{abstract}  We show that a reductive group scheme over a base scheme $S$
 admits a faithful linear representation if and only if its radical torus
 is isotrivial; that is, it splits  after a finite \'etale cover.
  
\end{abstract}

\maketitle

\vskip-4mm

\hfill \hfill {\it to the $75$-th anniversary of Gopal Prasad} \hfill

\bigskip

\bigskip


\vskip-4mm


\bigskip

\bigskip

\section{Introduction}

Let $S$ be a scheme. Let $\gG$ be an  
 $S$--group scheme. It is natural to ask 
 whether $\gG$ is linear; that is,  there exists a group monomorphism 
 $\gG \to \GL(\cE)$ where 
 $\cE$ is a locally free $\cO_S$-module of finite rank. In particular, $\gG$ admits a faithful representation on $\cE$.
This holds for affine algebraic groups over a field 
\cite[II, \S 2.3.3]{DG}.

In the case $S$ is locally noetherian and $G$ is of multiplicative type
of constant type and of finite type, Grothendieck has shown that 
$G$ is linear if and if $G$ is isotrivial, i.e. $G$ is split by a finite \'etale extension of $S$ \cite[XI.4.6]{SGA3}. In particular there exist non linear tori of rank $2$ over the local ring (at a node) of a nodal 
algebraic  curve  ({\it ibid}, X.1.6).
Firstly we extend that criterion over an arbitrary base  by using
Azumaya and separable $\cO_S$--algebras (Theorem \ref{thm_torus}).

Secondly we deal with the case $\gG$ reductive; that 
is, $\gG$ is smooth affine with reductive (connected) geometric fibers.
In this case a faithful representation is necessarily a closed immersion \cite[XVI.1.5]{SGA3}.
Positive results on the linearity  question are due to 
M. Raynaud  \cite[VI$_B$]{SGA3} and  R. Thomason \cite[3.1]{T}
which is essentially the implication $(i) \Longrightarrow (ii)$ 
in the theorem below.

%We shall see that the existence of a faithful linear representation
%does not hold in general even for rank two tori.
We can restrict our attention to the case 
when $\gG$ is of constant type (recall that the type is a locally
constant function on $S$);
this implies that  there exists a Chevalley $\ZZ$--group scheme $G$
such that $\gG$ is locally isomorphic to $G_S$ for the \'etale topology 
\cite[XXII.2.3, 2.5]{SGA3}. 
A short version of our main result is the following.

\begin{stheorem}\label{thm_main}
The following are equivalent:

\smallskip
(i) The radical torus $\rad(\gG)$  is isotrivial;

\smallskip

(ii) $\gG$ is linear.

\smallskip

\noindent Furthermore if $S$ is affine,  the above are  equivalent to 

\smallskip

(ii') there  exists  a
closed immersion  $i: \gG  \to \GL_n$ with $n \geq 1$
which is a homomorphism.

\end{stheorem}

We recall that  $\rad(\gG)$ is the maximal central subtorus of $\gG$ \cite[XXIV.4.3.6]{SGA3} and that (i)
means that $\rad(\gG)$ splits 
after passing to a finite \'etale cover $S'\to S$.
In the noetherian setting,  a variant of the implication $(i) \Longrightarrow (ii)$
has been shown by Margaux who furthermore  provided  an $\Aut(\gG)$-equivariant representation \cite{M}.
Note that condition (i) depends only on the quasi-split form of 
$\gG$ and also that it is always satisfied
 if $\gG$ is semisimple or if $\rad(\gG)$ is of rank one.
Furthermore, if  $S$ is a semilocal scheme, 
Demazure's characterization of isotrivial group schemes \cite[XXIV.3.5]{SGA3} 
permits us to deduce that 
 $\gG$ is isotrivial if and only if $\gG$ is linear, 
see section \ref{section_semilocal}.

Finally for $S=\Spec(R)$ with $R$ noetherian, we complete Thomason's approach
by showing that linearity for $\gG$ is equivalent to the 
resolution property (Th. \ref{thm_re}).

\medskip

\smallskip

\noindent{\bf Acknowledgements.} I thank Vladimir Chernousov, Laurent Moret-Bailly, Erhard Neher, Arturo Pianzola,
and Anastasia Stavrova for their valuable suggestions.
I thank the referee for a simplification of the proof of Proposition \ref{prop_torus}.

\section{Definitions and basic facts}

\subsection{Notation}
We use mainly  the terminology and notation of Grothendieck-Dieudonn\'e \cite[\S 9.4  and 9.6]{EGA1}
 which agrees with that  of Demazure-Grothendieck used in \cite[Exp. I.4]{SGA3}.
 
 \smallskip

\noindent (a) Let $S$ be a scheme and let $\cE$ be a quasi-coherent sheaf over $S$.
 For each morphism  $f:T \to S$, 
we denote by $\cE_{T}=f^*(\cE)$ the inverse image of $\cE$ 
by the morphism $f$.
 We denote by $\VV(\cE)$ the affine $S$--scheme defined by 
$\VV(\cE)=\Spec\bigl( \mathrm{Sym}^\bullet(\cE)\bigr)$;
it represents the $S$--functor $Y \mapsto \Hom_{\cO_Y}(\cE_{Y}, \cO_Y)$ 
\cite[9.4.9]{EGA1}. 
%This construction generalizes to algebraic spaces. 

\smallskip

\noindent (b) We assume now that $\cE$ is locally free of finite rank and denote by $\cE^\vee$ its dual.
In this case the affine $S$--scheme $\VV(\cE)$ is  of finite presentation 
(ibid, 9.4.11); also
the $S$--functor $Y \mapsto H^0(Y, \cE_{Y})= 
\Hom_{\cO_Y}(\cO_Y, \cE_{Y} )$ 
is representable by the  affine $S$--scheme $\VV(\cE^\vee)$
which is also denoted by  $\WW(\cE)$  \cite[I.4.6]{SGA3}.
 %Once again this construction generalizes to algebraic spaces.
 
The above applies to the locally free quasi-coherent sheaf
${\cE}nd(\cE) = \cE^\vee \otimes_{\cO_S} \cE$ 
 over $S$ so that we can consider
the affine $S$--scheme $\VV\bigl({\cE}nd(\cE)\bigr)$
which is an $S$--functor in associative  and unital algebras
\cite[9.6.2]{EGA1}.
Now we consider the $S$--functor $Y \mapsto \Aut_{\cO_Y}(\cE_{Y})$.
It is representable by an open $S$--subscheme of $\VV\bigl({\cE}nd(\cE)\bigr)$
which is denoted by $\GL(\cE)$ ({\it loc. cit.}, 9.6.4).

\smallskip

\noindent (c)  If $\cB$ is a locally free $\cO_S$--algebra (unital, associative) of finite rank,
we recall that the functor of invertible elements  of $\cB$ is representable
by an affine $S$-group scheme which is denoted by $\GL_1(\cB)$ \cite[2.4.2.1]{CF}.

For separable and Azumaya algebras, we refer to \cite{KO}.
Note that in \cite[\S 2.5.1]{CF}, separable algebras are supposed furthermore to be locally free of finite rank.

If  $\cB$  is a separable $\cO_S$--algebra which is a
locally free $\cO_S$--algebra of finite rank,
then  $\GL_1(\cB)$ it is a reductive $S$--group scheme  \cite[3.1.0.50]{CF}.

\smallskip

\noindent (d) We use the theory and terminology of tori and multiplicative group schemes of \cite{SGA3}; see also Oesterl\'e's survey \cite{O}.

\subsection{Finite \'etale covers}

The next lemma  is a consequence of the equivalence of categories describing 
finite \'etale $\cO_S$--algebra of rank $N$
\cite[\S 2.5.2]{CF}; it admits a simple direct proof.

\begin{slemma} \label{lem_brian} Let $N$ be a positive integer and let
 $\cC$ be a finite \'etale $\cO_S$--algebra of rank $N$. Then 
 there exists a finite \'etale cover  $T$ of $S$
 of degree  $N!$ such that 
 $\cC \otimes_{\cO_S} {\cO_T} \simlgr (\cO_T)^N $. 
\end{slemma}

\begin{proof}
We proceed by induction on $N$, the case $N=1$ being obvious.
We put $S'= \Spec(\cO_C)$, 
 this is a finite \'etale cover of $S$ of degree $N$. 
Since the diagonal map $S' \to S' \times_S S'$ is closed and
open \cite[$_4$.17.4.2]{EGA4}  
 there exists a decomposition  $\cC \otimes_{\cO_S} \cO_S'=\cO_{S'} \times  \cC'$ where $\cC'$ is a finite \'etale $\cO_{S'}$--algebra of rank $N-1$.
 Applying the induction process to $\cC'$ provides 
 a finite \'etale cover  $T$ of $S'$
 of degree  $(N-1)!$ such that 
 $\cC' \otimes_{\cO_{S'}} {\cO_T} \cong (\cO_T)^{N-1}$.
 Thus $\cC \otimes_{\cO_S} \cO_T=\cO_{T} \times  (\cO_T)^{N-1}$
 and $T$ is a finite \'etale cover of $S$ of degree $N!= N \times (N-1)!$.
\end{proof}

\subsection{Isotriviality} \cite[XXIV.4]{SGA3}
Let $\cH$ be a fppf $S$--sheaf in groups and let 
$\cX$ be a $\cH$--torsor. We say that  $\cX$ is {\it isotrivial} 
if there exists a finite \'etale cover $S'$ of $S$
which trivializes $\cX$; that is, satisfying $\cX(S') \not = \emptyset$.

The notion of locally isotrivial (with respect to the Zariski topology) is then clear and
there is also the following variant  of  {\it semilocally isotrivial}.

We say that $\cX$ is {\it  semilocally isotrivial}
if for each subset $\{s_1,\dots, s_n\}$ of points of $S$
contained in an affine open subset of $S$, there exists
an open subscheme $U$ of  $S$ containing $s_1, \dots, s_n$
such that $\cX \times_S U$ is isotrivial over $U$.

\medskip

A reductive $S$--group scheme $\gG$ is {\it isotrivial} if it is split by
a finite \'etale cover $S'$ of $S$. 
An isotrivial  reductive $S$--group scheme $\gG$ is necessarily of constant type.
If $\gG$ is of constant type with underlying Chevalley group scheme $G$,
$\gG$ is isotrivial if and only if
the $\Aut(G)$--torsor $\mathrm{Isom}( G_S, \gG)$ is isotrivial.

\subsection{Rank one tori}\label{subsec-rank1}
 The simplest case is that of $G=\GG_{m,S}$, the split 
$S$--torus of rank $1$. 
The $S$--functor $S' \mapsto  \Hom_{S'-gp}( \GG_{m,S'}, \GG_{m,S'})$
is representable by the constant $S$--group scheme $\ZZ_S$
\cite[VIII.1.5]{SGA3}. It follows that 
the $S$--functor $S' \mapsto  \Isom_{S'-gp}( \GG_{m,S'}, \GG_{m,S'})$
is representable by the constant $S$--group scheme $(\ZZ/2\ZZ)_S
=\Aut_{S-gp}(\ZZ_S)$.
On the other hand, $(\ZZ/2\ZZ)_S$ is the automorphism group 
of the split  \'etale cover $S \sqcup S \to S$ of degree $2$.
By definition an $S$--torus of rank one is a form 
of  $\GG_m$ for the  fpqc topology;
in the other hand, a degree $2$ \'etale cover 
is a   form  of $S \sqcup S$ for the finite \'etale topology
(see for example Lemma \ref{lem_brian}) and a fortiori 
for the  fpqc topology.

According to the faithfully flat descent technique (e.g. \cite[XXIV.1.17]{SGA3}), 
 there is then an equivalence of categories between
the groupoid of rank one tori over $S$ (resp.\ the groupoid of degree $2$ \'etale covers of $S$) and
the groupoid of $(\ZZ/2\ZZ)_S$-torsors.

More precisely one  associates to an $S$-torus $T$  of rank $1$
(resp.\ to an \'etale cover $E$ of degree $2$)  the 
$(\ZZ/2\ZZ)_S$-torsor $\mathrm{Isom}_{gr}( \GG_m, T)$
(resp. $\mathrm{Isom}( S \sqcup S , E)$). 
The inverse map is given by twisting the split object 
by a given $(\ZZ/2\ZZ)_S$-torsor.

Let $T$ be an $S$--torus of rank one and let $S'\to T$ be its 
associated  \'etale cover of degree $2$.
Then $T$ is splits after base change to $S'$ and so is 
isotrivial.  Furthermore  we claim that  $T$ is isomorphic to the quotient 
$Q=R_{S'/S}(\GG_{m,S'})/ \GG_m$ where $R_{S'/S}(\GG_{m,S'})$
stands for the Weil restriction.
We consider the  $(\ZZ/2\ZZ)_S$--equivariant exact sequences
$$
1 \to \GG_{m,S} \xrightarrow{\Delta} \GG_{m,S} \times \GG_{m,S}
\xrightarrow{\Pi} \GG_{m,S} \to 1
$$
where $\Pi(x,y)=x \, y^{-1}$ and the $\ZZ/2\ZZ$-action on $\GG_{m,S} \times \GG_{m,S}$ is by $(x,y) \mapsto (y,x)$.
The $\ZZ/2\ZZ$-action on the last factor is then by $x \mapsto  x^{-1}$.
Twisting this sequence by the $(\ZZ/2\ZZ)_S$-torsor $\mathrm{Isom}( S \sqcup S , S')$ yields the desired exact sequence
$$
1 \to \GG_m \to R_{S'/S}(\GG_{m,S'}) \to Q \to 1 
$$
where the identification of the second term is left to the reader
(it  is similar to that of the proof of \cite[XXIX.3.13]{SGA3}).

\medskip

\subsection{Linear representations}

Let $\gG$ be an $S$--group scheme.  We say that  $\gG$ is {\it linear} if there exists
a locally free $\cO_S$-module $\cE$ is of finite rank
and a group homomorphism $\gG \to \GL(\cE)$ which is a monomorphism.

The notion of locally linear $S$--group scheme  is then clear and
there is also the following variant  of  {\it semilocally linear}.

We say that $\gG$ is {\it  semilocally linear}
if for each subset $\{ s_1,\dots, s_n\}$ of points of $S$
contained in an affine open subset of $S$, there exists
an open subscheme $U$ of  $S$ containing $s_1, \dots, s_n$
such that $\gG  \times_S U$ is linear  over $U$.

\begin{slemma}\label{lem_affine}
Assume that $S=\Spec(R)$ and let $\cE$ be  a locally free  $\cO_S$--module  of finite rank. 
Then $\GL(\cE)$ embeds as a closed $S$--subgroup scheme in $\GL_n$ for some $n \geq 1$.
\end{slemma}

 \begin{proof} Since the rank of $\cE$ is a locally constant function \cite[Ch. \, 0, 5.4.1]{EGA1},
we can assume that  $\cE$ is locally free of constant of rank $r$.
Then $E=H^0(R, \cE)$ is a locally free $R$--module of rank $r$,
so is finitely generated projective \cite[Tag 00NX]{St}. It follows that there exists an integer
$n \geq 1$ and a decomposition $R^n =E \oplus E'$.
The homomorphism $\GL(\cE) \to \GL_n$ is a closed immersion. 
\end{proof}

\begin{slemma}\label{lem_weil}
Let $\gG$ be an $S$--group scheme
and let $S'$ be a finite locally free cover of $S$.
Then $\gG$ is linear if and only if $\gG \times_S S'$
is linear.
\end{slemma}

\begin{proof} We denote by $p: S' \to S$ the structure map.
If $\gG$ is linear, then $\gG \times_S S'$ is linear.
Conversely we assume that there exists a monomorphism
$i: \gG \times_S S' \to \GL(\cE')$ where 
 $\cE'$ is  a locally free  $\cO_{S'}$--module  of finite rank. We put  $\cE= p_*( \cE')$, this  is 
a  locally free $\cO_S$--module of finite rank.
 We consider the sequence of $S$--functors in $S$--groups
 $$
 \gG \to  R_{S'/S}( \gG \times_S S') \xrightarrow{R_{S'/S}(i)} R_{S'/S}( \GL(\cE') ) \to \GL( \cE)
 $$
 where $R_{S'/S}$ stands for the Weil restriction
 and the first map is the diagonal map which is a monomorphism.
 Since the Weil restriction for $S'/S$ transforms monomorphisms into monomorphisms,
 the map $R_{S'/S}(i)$ is also a monomorphism and 
  so is the last map  since 
 $R_{S'/S}( \GL(\cE') )(T) \subset  \GL( \cE)(T)$
 corresponds to automorphisms of $\cE \otimes_{\cO_S} \cO_T$
 which are $\cO_{S'} \otimes_{\cO_S} \cO_T$-linear.
 Since all maps are monomorphisms, we conclude that $\gG$ is linear.
\end{proof}

\section{Tori and group of multiplicative type}\label{section_torus}

\subsection{Maximal tori of linear groups}
Let $\cA$ be an Azumaya $\cO_S$--algebra.
We consider the reductive $S$--group scheme $\GL_1(\cA)$.
Let  $\cB \subset \cA$ be  a separable $\cO_S$--subalgebra of $\cA$
which is a locally free $\cO_S$--module of finite rank and
which is locally a direct summand of $\cA$ as $\cO_S$--module.
We get a monomorphism of reductive $S$-group schemes
$\GL_1(\cB) \to \GL_1(\cA)$.
In particular, if $\cB$ is commutative, $\cB$ is a finite 
\'etale $\cO_S$--algebra and $\GL_1(\cB)$ is a torus.
We come now to Grothendieck's definition of maximal \'etale subalgebras.

\begin{sdefinition} \cite[D\'ef. 5.6]{G}.
We say that a finite \'etale $\cO_S$--subalgebra  $\cC \subset \cA$ is 
{\it maximal} if $\cC$ is locally a direct summand of $\cA$ as $\cO_S$--module
and  if  the rank of $\cC \otimes_{\cO_S} \kappa(s)$  is the degree of $\cA_s \otimes_{\cO_S} \kappa(s)$
for each $s \in S$. 
\end{sdefinition}

If $\cC \subset \cA$ is  maximal finite \'etale $\cO_S$--subalgebra of 
$\cA$,  then the torus $\GL_1(\cC)$ is a maximal torus of $\GL_1(\cA)$
since it is the case on geometric fibers.
According to \cite[\S 7.5]{G}, all maximal $S$--tori of $\GL_1(\cA)$ occur in that manner;
this is part (3) of the following enlarged statement.

\begin{sproposition} \label{prop_torus}
 Let $\gS$ be a subgroup scheme of multiplicative type of $\GL_1(\cA)$
 and put $\cB=\cA^\gS$, the centralizer subalgebra of $\gS$.
 
 \smallskip
 
 \noindent (1) $\cB$ is a separable $\cO_S$--algebra which is locally a direct summand of $\cA$ as $\cO_S$--module.
 
 \smallskip
 
 \noindent (2) Let $\cC$ be the center of $\cB$; this is a
 finite \'etale $\cO_S$--algebra of positive rank which is locally a direct summand of $\cB$ (and $\cA$)
 as $\cO_S$--module.
 We have the closed immersions
 \[
  \gS \, \subset \, \GL_1(\cC)  \, \subset \, \GL_1(\cA).
 \]

 \smallskip
 
 \noindent (3)   If $\gS$ is a maximal torus of $\GL_1(\cA)$,
 then $\gS= \GL_1(\cC)$ and $\cC$ is a maximal finite \'etale $\cO_S$--subalgebra of 
 $\cA$.
 
 \smallskip
 
 \noindent (4) If $\gS$ is of constant type, then $\gS$ is isotrivial.
 
 \end{sproposition}

\begin{proof}
According to \cite[IX.2.5]{SGA3}, the map $\gS \to \GL_1(\cA)$ is a closed 
immersion so that  $\gS$ is of finite type over $S$.
 As a preliminary observation we notice that
 (1), (2), (3)  are local for the \'etale topology.
 We can assume that $S=\Spec(R)$, $\cA=\Mat_n(R)$ and that $\gS=D(M)$ for 
 $M$ an abelian group. Note that $M$ is finitely generated since  $\gS$ is
 of finite type ({\it ibid}, VII.2.1.b).

\smallskip

\noindent (1) We  consider the $M$--grading 
$R^n= \bigoplus_{m \in M} R^n_m$.
The $R$--modules $(R^n_m)_{m \in M}$ are finitely generated projective
so locally free of finite rank. 
There is a finite subset $M' \subset M$ such that 
$R^n= \bigoplus_{m \in M'} R^n_m$.
Then $\cB=  \prod_{m \in M'} \End_{R}( R^n_m)$
and each $\End_{R}(R^n_m)$ is a separable  $R$-algebra which is locally free of finite rank
\cite[III, example 2.8]{KO}.
Since a product of separable algebras is 
a separable algebra ({\it ibid}, III, proposition 1.7),
it follows that $\cB$ is a separable  $R$-algebra.
Furthermore $\cB$ is locally a direct summand of $\Mat_n(R)$
as $R$-module.

 \smallskip
 
 \noindent (2) Let $\cC$ be the center of $\cB$.
We have $R$--monomorphisms of groups  $\gS \, \subset \, \GL_1(\cC)  \, \subset \, \GL_1(\cA)$.
 
 \smallskip
 
 \noindent (3)  We assume that  $\gS$ is a maximal torus of $\GL_1(\cA)$
and want to establish that  $\gS= \GL_1(\cC)$. So $\GL_1(\cC)$ is an $R$--torus containing $\gT$ and since maximality
holds also in the naive sense \cite[XII.1.4]{SGA3},
we conclude that $\gT=\GL_1(\cC)$.

 \smallskip
 
 \noindent (4) 
 We come back to the initial setting (i.e. without localizing).
 We want to show that $\gS$ is isotrivial. 
According to \cite[ch.0, 5.4.1]{EGA1}, for each integer $l \geq 0$,
$S_l= {\bigl\{ s \in S \, \mid \, \rank( \cC_{\kappa(s)}) =l  \bigr\}}$ is an open subset of $S$
so that we have a decomposition in clopen  subschemes  $S=\coprod\limits_{l \geq 0} S_l$.
Without loss of generality we can assume that $\cC$ is locally free of rank $l$.
Lemma \ref{lem_brian}  provides a finite \'etale cover $S'$ of $S$
such that $\cC \otimes_{\cO_S} \cO_{S'} \cong (\cO_{S'})^l$.
Hence $\GL_1(\cC) \times_S S \cong  (\GG_m)^l \times_S S'$.
It follows that  $\gS \times_S S'$ is a subgroup $S'$--scheme of   $(\GG_m)^l \times_S S'$
of multiplicative type. 
According to \cite[IX.2.11.(i)]{SGA3} there exists a partition in clopen subsets
$S'= \sqcup_{i \in I} S'_i$ such that each  $\gS \times_S S'_i$ is diagonalizable.
Since  $\gS \times_S S'$ is  of  constant type we conclude that  $\gS \times_S S'$
is diagonalizable.
 \end{proof}

\subsection{Characterization of isotrivial groups of multiplicative type}

\begin{stheorem}\label{thm_torus} Let $\gS$ be an $S$-group scheme of multiplicative type of finite type and of  
constant type. Then the following are equivalent:

\smallskip

(i)  $\gS$ is isotrivial;

\smallskip

(ii) $\gS$ is linear;

\smallskip

(iii) $\gS$ is an $S$--subgroup scheme  of an $S$--group scheme $\GL_1(\cA)$
where $\cA$ is an Azumaya $\cO_S$--algebra.
 
\end{stheorem}

 \begin{proof}
 $(i) \Longrightarrow (ii)$. This follows from Lemma \ref{lem_weil}.

 \smallskip
 
 \noindent $(ii) \Longrightarrow (iii)$. By definition we have a monomorphism
 $\gS \to \GL(\cE)$ where $\cE$ is  a locally free  $\cO_{S}$--module  of finite rank. 
 Since $\End_{\cO_S}(\cE)$ is an Azumaya $\cO_S$--algebra, we get (iii).

 \smallskip
 
 \noindent $(iii) \Longrightarrow (i)$. This follows of Proposition \ref{prop_torus}.(4).
 \end{proof}

 \begin{sexamples} \label{ex_groth} {\rm (a) Grothendieck constructed
 a scheme $S$ and an $S$--torus $\gT$ which is locally trivial of rank $2$
  but which is not isotrivial \cite[\S X.1.6]{SGA3} (e.g. $S$ consists of
  two copies of the projective line over a field pinched at $0$ and $\infty$).
  Theorem \ref{thm_torus} shows that such an $S$--torus is not linear.
  \newline
  (b) Also there exists a  local ring $R$ and an $R$--torus $\gT$ of rank $2$ which is not isotrivial \cite[\S 1.6]{SGA3};
  the ring $R$  can be taken as the local ring of an algebraic curve at a double point.
Theorem \ref{thm_torus} shows that such an $R$--torus is not linear.
    }
\end{sexamples}

\section{Reductive case}

For stating the complete version of our main result,
we need more notation.
As in the introduction, $\gG$ is a reductive $S$--group 
scheme of constant type and $G$ is the underlying Chevalley
$\ZZ$--group scheme. 
We denote by $\Aut(G)$ the automorphism group scheme of $G$
and we have an exact sequence of $\ZZ$--group schemes \cite[th. XXIV.1.3]{SGA3}
$$
1 \to G_{ad} \to \Aut(G) \to \Out(G) \to 1.
$$
We remind the reader of the representability of the fppf sheaf
$\underline{\mathrm{Isom}}(G_S, \gG)$ by 
a $\Aut(G)_S$--torsor $\mathrm{Isom}(G_S, \gG)$ 
defined in \cite[XXIV.1.8]{SGA3}. The contracted product
$$\mathrm{Isomext}(G_S, \gG) := \mathrm{Isom}(G_S, \gG) 
\wedge^{\Aut(G)_S} \Out(G)_S$$ is an $\Out(G)_S$--torsor ({\it ibid}, 1.10)
which encodes the isomorphism class of the quasi-split form of $\gG$.

\begin{stheorem}\label{thm_main_complete}
The following are equivalent:

\smallskip

(i) The torus $\rad(\gG)$  is isotrivial;

\smallskip 

(ii) the $\Out(G)_S$--torsor $\Isomext(G_S, \gG)$ is isotrivial;

\smallskip

(iii) $\gG$ is linear;

\smallskip

(iv) $\rad(\gG)$ is linear.

\smallskip

\noindent Furthermore if $S$ is affine, we can take a faithful linear representation in some $\GL_n$ for $(iii)$ and $(iv)$.
\end{stheorem}

\begin{proof}$(i) \Longrightarrow (ii)$. We assume that $\rad(\gG)$
 is isotrivial and want to show that the $\Out(G)_S$--torsor $\cF=\Isomext(G_S, \gG)$ is isotrivial.
 In other words, we want to show that there exists a finite \'etale cover $S'$ of 
 $S$ such that $\gG \times_S S'$ is an inner form of $G$.
 
 Without loss of generality we can assume that $\rad(\gG)$ is a split torus.
 We quote now \cite[XXIV.2.16]{SGA3} for the Chevalley group $G$ over $\ZZ$
 which introduces  the $\ZZ$-group scheme
 $$H= \ker\bigl( \Aut(G) \to \Aut(\rad(G) \bigr);$$
 furthermore there is an equivalence of categories
 between the category of \break $H$-torsors over $S$
 and the category of pairs $(\gM, \phi)$ where $\gM$ is an $S$--form
 of $G$ and $\phi: \rad(G)_S \simlgr \rad(\gM)$.
 
 Since $\rad(\gG)$ is split, we choose an isomorphism  $\phi: \rad(G)_S \simlgr \rad(\gG)$ 
 and consider an $H$-torsor $\gP$ mapping to an object isomorphic to 
 $(\gG,\phi)$. Furthermore the quoted reference provides an exact sequence
 of $\ZZ$--group schemes
 $$
 1 \to G_{ad} \to H  \xrightarrow{p} F \to 1
 $$
 where $F$ is finite \'etale over $\ZZ$, so is constant.
 We denote by $S'= \gM \wedge^{H_S} F_S$ the contracted product of $\gM$ and $F_S$ with respect to $H_S$;
 this is an $F$--torsor over $S$,  hence is  a finite \'etale cover of $S$.
 It follows that $\gP \times_S S'$ admits a reduction to 
 a $G_{ad,S'}$--torsor $\gQ'$. Since the map $G_{ad} \to H \to \Aut(G)$
 is the canonical map, we conclude that $\gG_{S'} \cong \,   ^{\gQ'}\!\!G$ is an inner  form of $G$.

 \smallskip

 \noindent $(ii) \Longrightarrow (iii)$. Our assumption is that there 
 exists a finite \'etale cover $S'/S$ which splits 
 the $\Out(G)$--torsor $\mathrm{Isomext}(G_S, \gG)$.
 Lemma \ref{lem_weil} permits us to replace $S$ by $S'$, so we can assume
that  the $\Out(G)$--torsor $\mathrm{Isomext}(G_S, \gG)$ is trivial;
 that is, $\gG$ is an inner form of $G$.
 There exists a $G_{ad}$--torsor $\gQ$ over $S$
 such that $\gG \cong \, {^\gQ \! G}$. 
 Since $G \rtimes G_{ad}$ is defined over $\ZZ$, it admits
 a faithful representation $\rho: G \rtimes G_{ad} \to \GL_n$ over $\ZZ$ \cite[\S 1.4.5]{BT}.
 The map $\rho$ is then $G_{ad}$--equivariant and can be twisted
 by the $G_{ad}$--torsor $\gQ$. We obtain a faithful representation
 ${^\gQ \! G} \rtimes {^\gQ \! G_{ad}}  \to {^\gQ \! \GL_n}= \GL(\cE)$ where
 $\cE$ is the locally free $\cO_S$--module of rank $n$
 which is the twist of $(\cO_S)^n$ by the $\GL_n$--torsor
 $\gQ  \wedge^{G_{ad}} \GL_n$. Thus ${^\gQ \! G}$ is linear. 
 
  \smallskip

 \noindent $(iii) \Longrightarrow (iv)$. Obvious.
 
 \smallskip

 \noindent $(iv) \Longrightarrow (i)$. 
 Since  $\rad(\gG)$ is a form of $\rad(G)$, it is of constant rank
  and Theorem \ref{thm_torus} shows that $\rad(G)$ is isotrivial.
  
\smallskip
  
Finally the refinement for  $S$ affine follows from Lemma \ref{lem_affine}.  
\end{proof}

 \begin{scorollary}\label{cor_main}
 Under the assumptions of $\gG$, let $\gG^{qs}$ be the quasi-split form of 
 $\gG$.
 Then $\gG$ is linear if and only if $\gG^{qs}$ is
 linear.
 \end{scorollary}

The next corollary slightly generalizes  a result by Thomason \cite[cor. 3.2]{T}.

\begin{scorollary}\label{cor_main2} Assume that either

\smallskip

\noindent (i) $S$ is locally noetherian and  geometrically unibranch
(e.g. normal);

\smallskip

\noindent (ii) $\rad(\gG)$ is of rank $\leq 1$ (in particular if $G$ is semisimple).

\smallskip

\noindent Then $\gG$ is linear.
\end{scorollary}

\begin{proof}
 In case (i),  the torus $\rad(\gG)$ is isotrivial \cite[X.5.16]{SGA3}.
 In case (ii), we have $\rad(G)=1$ or $\GG_m$ (since $G$ is split), so 
 that $\rad(\gG)$ is
 split by a quadratic \'etale cover of $S$ (\S \ref{subsec-rank1}), hence is isotrivial.
 Hence Theorem \ref{thm_main} implies that $\gG$ is linear.
\end{proof}

The next corollary extends
 Demazure's characterization of locally isotrivial 
 reductive group schemes \cite[XXIV.4.1.5]{SGA3}.
 
 \newpage

\begin{scorollary}\label{cor_main_complete}
The following are equivalent:

\smallskip 

(i) $\gG$ is locally (resp.\ semilocally) isotrivial;

\smallskip

(ii) The torus $\rad(\gG)$  is locally (resp.\ semilocally) isotrivial;

\smallskip 

(iii) the $\Out(G)_S$--torsor $\Isomext(G_S, \gG)$ is locally (resp.\ semilocally) isotrivial;

\smallskip

(iv) $\gG$ is locally (resp.\ semilocally) linear;

\smallskip

(v) $\rad(\gG)$ is locally (resp.\ semilocally) linear.

\end{scorollary}

\begin{proof}
 In view of Theorem \ref{thm_main_complete}, it remains
 to establish the equivalence $(i) \Longleftrightarrow (ii)$.
 Since this is  precisely the quoted result \cite[XXIV.3.5]{SGA3}, the proof is complete.
\end{proof}

\begin{scorollary}\label{cor_main_complete2}
Let $\gH$ be a reductive $S$--subgroup scheme of $\gG$.
If  $\gG$ is locally (resp.\ semilocally) isotrivial, then 
$\gH$ is locally (resp.\ semilocally) isotrivial.
\end{scorollary}

\begin{proof}
 Corollary \ref{cor_main_complete} shows that $\gG$ is 
  locally (resp.\ semilocally) linear  and so is $\gH$.
  Therefore $\gH$ is locally (resp.\ semilocally) isotrivial. 
\end{proof}

\smallskip

\section{The semilocal case} \label{section_semilocal}

We assume that $S=\Spec(R)$ where $R$ is a semilocal ring
and continue to assume that the reductive $S$--group scheme
$\gG$ is of constant type.
We remind the reader that $\gG$ admits a maximal  torus
(Grothendieck, \cite[XIV.3.20 and footnote]{SGA3}).

\begin{scorollary} \label{cor_semilocal} Let $\gT$ be a maximal torus of $\gG$.
The following are equivalent:

\smallskip 

(i) $\gG$ is isotrivial;

\smallskip

(ii) The torus $\rad(\gG)$  is isotrivial;

\smallskip 

(iii) the $\Out(G)_S$--torsor $\Isomext(G_S, \gG)$ is isotrivial;

\smallskip

(iv) $\gG$ is  linear;

\smallskip

(v) $\rad(\gG)$ is  linear;

\smallskip

(vi)  $\gT$ is linear;

\smallskip

(vii)  $\gT$ is isotrivial.

\end{scorollary}
 
\begin{proof}
From Corollary \ref{cor_main_complete}, we have the equivalences
 $(i) \Longleftrightarrow (ii) \Longleftrightarrow (iii)  \Longleftrightarrow (iv)
  \Longleftrightarrow (v)$. On the other hand,
  the equivalence  $(vi)  \Longleftrightarrow (vii)$ holds
  according to Theorem \ref{thm_torus}. Now we observe that 
  the implications $(iv) \Longrightarrow (vi)$ and
  $(vi) \Longrightarrow (v)$ are obvious so the proof is complete.
\end{proof}

\section{Equivariant resolution property}

\begin{sdefinition} Let $\gG$ be a flat  group scheme over $S$ acting on a locally noetherian $S$--scheme
$\gX$. One says that $(\gG, S, \gX)$ has the
resolution property ( $(RE)$ for short)  if for every coherent $\gG$-module $\cF$ on $\gX$, there
is a locally free coherent $\gG$-module ( i.e. a $\gG$-vector bundle $\cE$) and a
$\gG$-equivariant epimorphism $\cE \to \cF \to 0$.
\end{sdefinition}

We strengthen Thomason's results.

\begin{stheorem} \label{thm_re} Let $\gG$ be a reductive $S$--group scheme.
We assume that $S$ is separated noetherian and that
$(1,S,S)$ satisfies the resolution property, e.g. $S$ is affine or regular or
admits an ample family of line bundles.
Then the following are equivalent:

\smallskip

(i) $\gG$ is linear;

\smallskip

(ii) $\rad(\gG)$ is isotrivial;

\smallskip

(iii) $\gG$ satisfies $(RE)$.

\end{stheorem}

\begin{proof}
$(i) \Longleftrightarrow (ii)$. This is a special case of Theorem \ref{thm_main_complete}.

\smallskip

\noindent $(ii) \Longrightarrow (iii)$. This is Thomason's result \cite[Theorem 2.18]{T}.

\smallskip

\noindent $(iii) \Longrightarrow (i)$.  This is Thomason's result \cite[Theorem 3.1]{T}, see also \cite[VI$_B$.13.5]{SGA3}.
\end{proof}

\begin{sremark}{\rm Example \ref{ex_groth}.(b) is  an example of 
a local noetherian ring $R$ and of a rank two non-isotrivial torus $\gT$.
Theorem \ref{thm_re} shows that $\gT$ does not satisfy (RE).
This answers a question of Thomason \cite[\S 2.3]{T}.
 }
\end{sremark}

\bigskip

\medskip

\end{document}